\documentclass{amsart}
\usepackage{amsmath,latexsym}
\usepackage[psamsfonts]{amssymb}
\usepackage{times}
\usepackage[mathcal]{euscript}
\usepackage{amsfonts}
\usepackage{amssymb}
\usepackage{amsmath}
\usepackage{amsthm}
\usepackage{bm}
\usepackage[english]{babel}
\usepackage[bookmarks]{hyperref}
\numberwithin{equation}{section} \textwidth=140mm \textheight=200mm
\newcommand{\bbC}{\mathbb C}

\renewcommand{\epsilon}{\varepsilon}

\newcommand{\be}{\begin{equation}}
\newcommand{\ee}{\end{equation}}

\newcommand{\R}{\mathbb{R}}

\newcommand{\T}{\mathbb{T}}

\newcommand{\Z}{\mathbb{Z}}

{\bf}{\it}
\newtheorem{theorem}{Theorem}[section]
\newtheorem{lemma}[theorem]{Lemma}
\newtheorem{corollary}[theorem]{Corollary}
\newtheorem{hypothesis}[theorem]{Hypothesis}
\newtheorem{definition}[theorem]{Definition}

\newtheorem{remark}[theorem]{Remark}

\date{\today}

\begin{document}

\title[Convergent expansions of eigenvalue of the generalized Friedrichs ...]
{Convergent expansions of eigenvalues of the generalized Friedrichs
model with a rank-one perturbation}

\author{Saidakhmat N. Lakaev $^{1}$}
\address{
$^1$ Samarkand State University, Samarkand (Uzbekistan)}
\email{slakaev@mail.ru}
\address{{$^{1}$ Department of Mathematics, Samarkand State University,
Samarkand (Uzbekistan)} \\ {E-mail: slakaev@mail.ru }}

\author{Shakhzod Kurbanov$^{2}$}
\address{$^{2}$ {Department of Mathematics, Samarkand State University,
Samarkand (Uz\-be\-kis\-tan)}\\ {E-mail:kurbanov-shaxzod@mail.ru}}

\maketitle
\begin{abstract}
We study the existence of eigenvalues of the generalized Friedrichs
model $H_\mu(p)$, with a rank-one perturbation, depending on
parameters $\mu>0$ and $p\in\T^2$, and found an absolutely
convergent expansions for eigenvalues at $\mu(p)$, the coupling
constant threshold. The expansions are highly dependent on that,
whether the threshold $m(p)$ of the essential spectrum is: $(i)$
neither an threshold eigenvalue nor a threshold resonance; $(ii)$ a
threshold resonance; $(iii)$ an threshold eigenvalue.
\end{abstract}
\subjclass{2010 Mathematics Subject Classification.  Primary: 81Q10,
Secondary:  47A10}

\textit{Keywords and phrases: generalized Friedrichs model, coupling
constant threshold, hamiltonian,  dispersion relation, threshold
resonance, threshold eigenvalue}

\section{Introduction}
We consider a family of the generalized Friedrichs models
$H_\mu(p),$ $\mu>0,$ $p\in\T^d,\,d\ge1$ with a rank-one
perturbation, which is a generalization of the family of
two-particle Schr\"{o}dinger operators $H_{\mu}(k),$
$k\in\T^d=(-\pi,\pi]^d$ associated to a system of two arbitrary
(identical) quantum mechanical particles moving on the
$d$-dimensional lattice $\Z^d,$ $d\ge1$ and interacting via
zero-range attractive or repulsive potential (see, e.g.,
\cite{GadellaPronko2011}, \cite{LDK2013}, \cite{CivitareseGadella},
\cite{Naboko} and references therein).

The description of Bose condensates is an example, where zero-range
interactions are basic to theories of the condensed aggregates.
Zero-range potentials are the mathematically correct tools for
describing contact interactions. The latter reflects the fact that
the zero-range potential is effective only in the s-wave
\cite{{Macek},{Yakovlev}}.

A large class of the {\it zero-range models} can be constructed
based on John von Neumann operator extension technique. It is
provided an extended insight into the Friedrichs model, as an
universal tool of the analytic perturbation theory and give a
state-of-art review of others fitted zero-range models and the
results can be considered as a motivation and a practical
introduction into the area of applied spectral analysis of linear
dynamical systems \cite{Pavlov}.

Under some assumptions on the operator one can obtain a meromorphic
continuation of the resolvent around a neighborhood of the
threshold, and then give a unified discussion of eigenvalues and
resonances. There are several results of this type in the
literature, see for example
\cite{{GesztesyHolden},{LakaevTilavova1994},{Rauch}}.

The authors of \cite{Holden} and \cite{Klaus-Simon80} studied for
the Schr\"{o}dinger ope\-ra\-tors $H_{\mu}=-\Delta+\mu V$  a
situation, where as $\mu$ approaches to $\mu_0\ge0$ an eigenvalue
$E(\mu)$ accumulates to $0$, the bottom of the essential spectrum of
$H_{\mu}$, i.e., as $\mu$ approaches to $\mu_0$ an eigenvalue is
absorbed at the threshold of continuum, and conversely, as $\mu$
seeks to $\mu_0+\epsilon, \epsilon>0,$ the continuum {\it gives
birth} to a new eigenvalue. This phenomenon in \cite{Klaus-Simon80}
is called {\it coupling constant threshold}.  Moreover, in
\cite{Holden}, \cite{Klaus-Simon80} an absolutely convergent
expansion for the eigenvalue $E(\mu)$ at $\mu_0\ge0$, the coupling
constant threshold of $H_{\mu}$, was found. In \cite{Jensen} some
results on the perturbation of eigenvalues embedded at thresholds in
a two channel model Hamiltonian with a small of-diagonal
perturbation, which are related to results on coupling constant
thresholds. Examples are given of the various types of behavior of
the eigenvalue under perturbation.

Results devoted to coupling constant threshold (c.c.th.) have also
been obtained on perturbations of Schr\"odinger operators with
periodic potentials \cite{{FassariKlaus}}, and for the Dirac
operator \cite{Klaus}.

The main object of the work \cite{FassariKlaus} is to study the
analytic behavior of eigenvalue $E(\lambda)$, when $\lambda$ is near
a c.c.th. $\lambda$. The authors  applied slightly different (non
self-adjoint)versions of Birman-Schwinger principle  depending on
the dimension $d=1,2,3$.

In \cite{Lakaev-Holmatov11} the existence of positive coupling
constant threshold $\mu=\mu(k)>0$ for the Schr\"odinger operator
$H_{\mu}(k),$ $k\in\T^d,\,d\geq3$ associated to a system of two
identical quantum mechanical particles (bosons) moving on the
lattice $\Z^d,$ $d\geq3$ and interacting via zero-range repulsive
potential is proved: the operator has no eigenvalues for any
$0<\mu<\mu(k)$, nevertheless for each $\mu>\mu(k)$ it has a unique
eigenvalue $E(\mu,k)$ lying above the essential spectrum. Moreover,
an absolutely convergent expansions for the eigenvalue $E(\mu,k)$ at
$\mu=\mu_0$ depending on $d\geq3$ was found. However, in
\cite{LKhLtmf12} the absence of positive coupling constant
thresholds, i.e., the existence for eash $\mu>0$ a unique eigenvalue
$E(\mu,k)$ of the discrete Schr\"odinger operator
$H_{\mu}(k),\,k\in\T^d,\,d=1,2,$ associated to a system of two
identical quantum-mechanical particles (bosons) on $\Z^d, d=1,2$ is
proved and an absolutely convergent expansion for $E(\mu,k)$ at
$\mu=0$ was found.

Notice that for the Schr\"odinger operators of a system of two
arbitrary particles moving on $\R^d$ or $\Z^d,$ $d\geq1$  the
coupling constant threshold vanishes, if $d=1,2$ and it is positive,
if $d\ge3$. Furthermore, for the Schr\"odinger operators of a system
of two identical particles moving on $\R^d$ or $\Z^d,d=1,2$ the
coupling constant threshold vanishes, if the particles are bosons,
however it is positive, if they are fermions (see, e.g.,
\cite{Klaus1976}, \cite{Simon1976}, \cite{Lakaev92},
\cite{LakaevTilavova1994}, \cite{CivitareseGadella}).

For a wide class of the two-particle discrete Schr\"odinger
operators $H(k),\,k\in\T^d$ on the $d$-dimensional lattice $\Z^d$,
$d\ge 3$, for all nonzero values of quasi-momentum $k$ the existence
of eigenvalues of $H(k)$ below the threshold, under the assumption
that $H_\mu(0)$ has either a threshold energy resonance or a
threshold eigenvalue at the threshold (bottom) of the essential
spectrum was proved \cite{{ALMM06}}. Similar result for the
Friedrichs model was obtained in \cite{ALzM07}.

In \cite{Nishida2013} a system of $3$ nonrelativistic spinless
fermions in $2$ dimensions, which interact through
spherically-symmetric pair interactions was considered. A claim has
been made for the existence of the so-called super Efimov effect.
Namely, if the interactions in the system are fine-tuned to a p-wave
resonance, an infinite number of bound states appears, whose
negative energies are scaled according to the double exponential
law.  The mathematical proof that such a system indeed has an
infinite number of bound levels is presented in \cite{Gridnev2014}.
It is also proved that $\lim_{E\rightarrow0} |\ln| \ln E||^{-1}N(E)
= 8/(3p)$, where $N(E)$ is the number of bound states with the
energy less than $-E < 0$.

The reasons to consider the family of the generalized Friedrichs
models associated to a system of two particles interacting via pair
local {\it attractive or repulsive potentials} are as follows:

(i) The family of generalized Friedrichs models generalizes and
involves some important behaviors of the Schr\"odingier operators
associated to the hamiltonian for systems of two arbitrary particles
moving on $\R^d$ or $\Z^d,d\geq 1,$ as well as, the hamiltonian for
systems of both bosons and fermions (see, e.g.,
\cite{LAK2012},\,\cite{Lakaev86});

(ii) The works (see, e.g., \cite{AbdullaevLakaev2003},
\cite{ALzM2004}, \cite{ALK2012}, \cite{Lakaev1993}) devoted to the
Efimov effect for $3$ quantum mechanical particles on the hypercubic
lattice $\Z^3$ gives an assurance on the existence of the super
Efimov effect for a system of $3$ spinless fermions on the two
dimensional hypercubic lattice $\Z^2$ interacting via short range
pair potentials if the interactions  was tuned in such a way that
pairs of fermions has no negative eigenvalues, but are at the
coupling constant threshold;

(iii) There are interesting features of the super Efimov effect,
that a system of $3$ spinless fermions on the lattice $\Z^2$ may
have an infinite number of bound states depending on the
quasi-momentum, although the same system in $3$ dimensions has at
most a finite number of levels with negative energy (see, e.g.,
\cite{ALzM2004}, \cite{ALK2012},\cite{ALDj2009}, \cite{Lakaev1993}).

We study, in the important case $d=2$ (see, e.g.,
\cite{Gridnev2014}, \cite{Nishida2013}), the coupling constant
thresholds for the generalized Friedrichs models $H_\mu(p),$
$p\in\T^2$, $\mu>0$, which are  generalizations (more general
dispersion relations) of the two-particle Schr\"{o}dinger operators
$H_{\mu}(k),$ $k\in\T^2$, $\mu>0$ on the lattice $\Z^2$ with the
local interactions.

In the current paper, we answer the following question:  what is the
character of convergence of eigenvalue $E(\mu,p)$ of the generalized
Friedrichs models $H_\mu(p),\,p\in\T^2$, $\mu>0$ to $m(p)=\inf
\sigma_{\rm ess}(H_{\mu}(p))$, the bottom
 of the essential spectrum
as $\mu\to \mu(p)\ge0$?

Furthermore, unlike to the cases
\cite{Klaus-Simon80},\cite{LKhLtmf12} and \cite{Lakaev-Holmatov11},
we found the absolutely convergent expansions (asymptotics) of the
eigenvalues $E(\mu,p)$ at the coupling constant thresholds
$\mu(p)\ge0$ for the operators, which are associated to a system of
two fermions (see,  $(ii)$ and $(iii)$ of Theorem \ref{main}).

Surprisingly, we derive absolutely convergent expansions
(asymptotics) of the eigenvalues $E(\mu,p)$ at $\mu(p)$, the
coupling constant threshold, in the cases, when the threshold $m(p)$
is: $(i)$ none an threshold eigenvalue or a threshold resonance;
$(ii)$ a threshold resonance; $(iii)$  an threshold eigenvalue (see
Theorem \ref{main}).

A family $H_\mu(p),$ $\mu>0,$ $p\in\mathbb{T}^d$ of the generalized
Friedrichs models with the local perturbation of rank one,
associated to a system of two particles, moving on the $d$-
dimensional lattice $\mathbb{Z}^d,$ was considered in
\cite{LAK2012},\cite{LDK2013} and \cite{LDD2019}. A criterion to
existence of a coupling constant threshold $\mu=\mu_0(p)\ge0$
depending on the parameters of the model was proved in
\cite{LAK2012} and \cite{LDD2019}. An absolutely convergent
expansion for the unique eigenvalue $E(\mu,p)$ of $H_\mu(p)$ at
$\mu(p)=0$ was found in \cite{LDK2013}.

\section{Preliminaries and Main Results.}
Let $\Z^2$ be the two-dimensional hypercubic lattice and
$\T^2=(\mathbb{R}/2\pi\mathbb{Z})^2=(-\pi,\pi]^2$ be the
two-dimensional torus (Brillion zone), the dual group of $\Z^2.$

Let $L^2(\T^2)$ be the Hilbert space of square-integrable functions
defined on the torus $\T^2 $ and $\varphi \in L^2(\T^2),$
$f_0\in\mathbb{C}.$

We define the operator $\Phi:L^2(\T^2)\to \mathbb{C}$ and its
adjoint $\Phi^*: \mathbb{C}\to L^2(\T^2)$ as
$$\Phi f=(f,\varphi)_{L^2(\T^2)} \,\mbox{and}\, \,\, \Phi^* f_0=\varphi(q) f_0,$$  where
$(\cdot,\cdot)_{{L}^{2}(\T^2)}$ is inner product in ${L}^{2}(\T^2).$

Let $H_0(p),$ $p\in\T^2$ is a multiplication operator by the
function $w_p(\cdot):=w(p,\cdot):$
\begin{equation}\nonumber\label{h0}
(H_0(p)f)(q)=w_p(q)f(q),\quad f\in L^2(\T^2).
\end{equation}
Then the generalized Friedrichs model $H_{\mu}(p),p\in\T^2$ is
defined in $L^2(\T^2)$ as follows:
\begin{equation*}\label{H}
H_{\mu}(p)=H_0(p)-\mu\Phi^*\Phi,\,\,\mu>0.
\end{equation*}

The perturbation $v=\Phi^{\ast}\Phi$ of $H_0(p), p\in\T^2$ is
positive operator of rank one. Consequently, by the well-known Weyl
theorem \cite[Theorem XIII.14]{RSIV78} on compact perturbations, the
essential spectrum of $H_\mu(p), p\in\T^2$ satisfies the equalities
$$
\sigma_{ess}(H_\mu(p))=\sigma_{ess}(H_0(p))=\sigma(H_0(p))
$$
and fills the segment $[m(p),\,M(p)]$ on the real axis, where
\begin{equation*}
m(p)=\min_{q\in \T^2}w_p(q),\quad M(p)= \max_{q\in \T^2}w_p(q).
\end{equation*}
\begin{remark}\label{lem3.1200}
We note that the positivity of $\Phi^{\ast}\Phi$ yields that the
operator $H_{\mu}(p)$ has none eigenvalue lying above $M(p)$.
\end{remark}

Throughout the paper we assume the following
\begin{hypothesis}\label{Hyp1}
Assume that the following conditions are satisfied:
\begin{itemize}
\item[(i)] the function $\varphi(\cdot)$ is nontrivial and real-analytic on $\T^2;$
\item[(ii)] the function $w(\cdot,\cdot)$ is real-analytic on
$(\T^2)^2=\T^2\times \T^2$ and has a unique non degenerated minimum
at $(0,0)\in (\T^2)^2$.
\end{itemize}
\end{hypothesis}
Hypothesis \ref{Hyp1} yields that there existence a
$\delta$-neighborhood $U_{\delta }(0)\subset \T^2$ of the point
$p=0\in \T^2$ and an analytic function $q_0:U_{\delta}(0)\to \T^2$
that for any $p\in U_{\delta}(0)$ the point $q_0(p)\in\T^2$ is a
unique non degenerated minimum of the function $w_p(\cdot)$.

For any $\mu>0$ and $p\in \T^2$ we define an analytic function
$\Delta(\mu,p; \cdot)$ (the Fredholm determinant, associated to the
operator $H_\mu(p)$)\,in $\mathbb{C}\setminus [m(p); M(p)]$ as
follows
\begin{equation*}\label{Det.H.lamb}
\Delta(\mu,p; \cdot)=1-\mu \Omega(p;\cdot),
\end{equation*}
where
\begin{equation}\label{omega}
\Omega(p;z)=\int\limits_{\T^2}
\frac{\varphi^2(s)ds}{w_p(s)-z},\qquad p\in \T^2, \quad z\in
\mathbb{C}\backslash [m(p); M(p)].
  \end{equation}
In our proofs  we apply results of Lemmas 3.1 and 3.8 of
\cite{LAK2012} and hence, for  conveniens of the readers, we recall
these results as the following lemma
\begin{lemma}\label{egenvalue}
Assume Hypothesis \ref{Hyp1}.
\begin{enumerate}
\item[{\rm (i)}]
A number $z\in \bbC\setminus \sigma_{ess}(H_{\mu }(p)),$ $p\in\T^2$
is an eigenvalue of the operator $H_{\mu}(p)$ if  and  only  if
$$\Delta(\mu,p; z)=0.$$ The corresponding  eigenfunction $f$ has
form
\begin{equation*}\label{f}
f_{\mu, p}(q)=\frac{C\mu\varphi(q)}{w_p(q)-z},
\end{equation*}
and is analytic on $\T^2,$ where $C=C(p)>0$ is the normalizing
constant.

\item[{\rm (ii)}]
Let $s=q_{0}(p),$ $p\in U_{\delta}(0)$ be a unique non degenerated
minimum point of the function $w_{p}(s)$ and $\varphi(q_{0}(p))=0$
resp. $\varphi(q_{0}(p))=\nabla\varphi(q_{0}(p))=0$. Then
\begin{align*}
\Delta(\mu,p;m(p))=
1-\mu\int\limits_{\T^2}\frac{\varphi^2(q)dq}{w_p(q)-m(p)}=0
\end{align*}
if and only if $z=m(p)$, the bottom of the essential spectrum
$\sigma_{ess}(H_{\mu }(p))$ is a threshold resonance resp. an
eigenvalue for the operator $H_{\mu}(p)$, $\mu>0$, i.e., the
equation \begin{equation*}\label{k} H_{\mu}(p)f=m(p)f\end{equation*}
has a nonzero solution
\begin{equation*}\label{kkk} f_{\mu,p}(\cdot)=\frac{C\mu
\varphi(\cdot)}{w_p(\cdot)-m(p)}\end{equation*} which belongs to
$L_1(\T^2)\backslash L_2(\T^2)$ resp. $L_2(\T^2)$, where $C=C(p)>0$.
\end{enumerate}
\end{lemma}


\begin{definition}\label{definition}
Define $\mu(p)>0$ as
\begin{equation}\label{integral}
\mu(p)=\left(\int\limits_{\T^2}
\frac{\varphi^2(s)ds}{w_p(s)-m(p)}\right)^{-1}>0,
\end{equation} if $\varphi(q_0(p))=0$ and $\mu(p)=0,$ if
$\varphi(q_0(p))\neq0,$ $p\in U_{\delta}(0).$ \end{definition}

\begin{remark} Note that in the case $\varphi (q_{0}(p))=0,$
$p\in U_{\delta}(0)$ the existence of the integral and hence
positive coupling constant in \ref{integral} is proven in
\cite{LAK2012}.
\end{remark}

In the next theorem we recall, for reading convenience, a criterion
for existence of a unique eigenvalue below $m(p)$, the bottom of the
essential spectrum of the operator $H_{\mu}(p),$ $p\in
U_{\delta}(0)$ (see, \cite[Theorem 2.3]{LAK2012}).

\begin{theorem} \label{last}
Assume Hypothesis \ref{Hyp1}.  Then for any fixed $p\in
U_{\delta}(0)$ the operator $H_\mu(p)$ has  a unique eigenvalue
$E(\mu,p)$  below $m(p)$, the bottom of the essential spectrum, if
and only if $\mu>\mu(p)$. Moreover: if $\varphi(q_0(p))=0$,
$\nabla\varphi(q_0(p))\neq0$ and $\mu=\mu(p)$, then the threshold
$z=m(p)$  is a virtual level of the operator $H_{\mu}(p)$, i.e. the
equation $H_{\mu}(p)f=m(p)f$ has a non-zero solution $f\in
L_1(\mathbb{T}^2)\backslash L_2(\mathbb{T}^2)$; if $\varphi
(q_0(p))=\nabla\varphi(q_0(p))=0$ and $\mu=\mu(p)$, then the number
$z=m(p)$ is an eigenvalue of the operator $H_{\mu}(p)$.
\end{theorem}

The main result of the current paper, is to found an absolutely
convergent expansions for the eigenvalue $E(\mu,p)$ at the coupling
constant threshold $\mu(p)\ge0$, in the cases, when the threshold
$m(p)$ is : $(i)$ none a threshold eigenvalue or a threshold
resonance; $(ii)$ a threshold resonance; $(iii)$ a threshold
eigenvalue.

\begin{theorem} \label{main}
Assume Hypothesis \ref{Hyp1}. Then for any fixed $p\in
U_{\delta}(0),$ $\mu$ tends to $\mu(p)$ if and only if $E(\mu,p)$
tends to
 the threshold $m(p)$. Moreover for any fixed $p\in U_{\delta}(0)$ and sufficiently small positive
$\mu-\mu(p)$ a unique eigenvalue $E(\mu,p)$ of the operator
$H_\mu(p)$ has the following absolutely convergent expansions:
\begin{enumerate}
\item[{\rm (i)}] Let $\varphi(q_0(p))\neq0$. Then  \begin{equation*}\label{as1}E(\mu,p)= m(p)-
a(p)e^{(\alpha_0(p)\mu)^{-1}}- \sum\limits_{m\geq1, n\geq 1,
m+n\geq3 }{c}(m,n)(p)\mu^{m}\tau^{n},\end{equation*}
$$\tau=\frac1\mu \,e^{(\alpha_0(p)\mu)^{-1}},\,\, a(p)=e^{-c_0(p)/\alpha_0(p)},\quad
 \alpha_0(p)<0$$ and $c_0(p),$ $c(n,m)(p),$
 $m,n=0,1,2,...$-- real numbers.

\item[{\rm (ii)}] Let $\varphi(q_0(p))=0$ and
$\nabla\varphi(q_0(p))\neq0$. Then \begin{equation*}\label{ass1}
E(\mu,p) =m(p)-
[\hat{\alpha}_1(p)\mu^2(p)]^{-1}\frac{\hat{\mu}}{\ln\hat{\mu}^{-1}}
-\sum\limits_{n\geq1, r\geq1, s\ge0, n+r+s\ge3} {c}(n,s,r)(p)\,
\tau^{n}\,\,\hat{\mu}^{r}\,\,\omega^{s}\,\,, \end{equation*} where $
\hat{\alpha}_1(p)>0$ is defined in \eqref{qqq1},\, $c(n,r,s)(p),\,\,
n,r,s=0,1,2...$--real numbers and
$$\tau=\frac{1}{\ln\hat{\mu}^{-1}}, \quad \omega=\frac{\ln \ln
\hat{\mu}^{-1}}{\ln \hat{\mu}^{-1}},\quad \hat{\mu}=\mu-\mu(p).$$
\item[{\rm (iii)}]
Let $\varphi(q_0(p))=\nabla\varphi(q_0(p))=0$. Then
\begin{equation}\label{ox}E(\mu,p)= m(p)-a(p)\hat{\mu}
-\sum\limits_{l\geq0,s\ge1, l+s\geq2}
\hat{c}(l,s)(p)\,\tau^{l}\,\hat{\mu}^{s}\end{equation}
$$
\tau=\hat{\mu}\ln\hat{\mu},\quad\hat{\mu}={\mu}-{\mu}(p),$$ where
$a(p)>0$ and $\hat{c}(l,s)(p),l,s=0,1,2,...$ are real numbers.
\end{enumerate}
\end{theorem}

The asymptotics of the eigenvalue $E(\mu,p)$ at  the coupling
constant threshold $\mu(p)\ge0$ of the essential spectrum are given
in the following corollary.

\begin{corollary}
Assume Hypothesis \ref{Hyp1}. For any fixed $p\in U_{\delta}(0)$ the
following asymptotics are hold:
\begin{enumerate}
\item[{\rm (i)}] If $\varphi(q_0(p))\neq0,$ then
 \begin{equation}\nonumber\label{s} m(p)-E(\mu,p)=
 e^{-c_0(p)/\alpha_0(p)}\sigma + O([\mu^2\tau]),\quad as\quad \mu\to 0,\end{equation}
$$\sigma=e^{(\alpha_0(p)\mu)^{-1}},\quad\tau=\frac1\mu{e^{(\alpha_0(p)\mu)^{-1}}},\quad
 \alpha_0(p)<0.$$
\item[(ii)]  If $\varphi(q_0(p))=0$ and $\nabla\varphi(q_0(p))\neq0,$
 then
$$
m(p)-E(\mu,p)
=[\hat{\alpha}_1(p)\mu^2(p)]^{-1}\frac{\hat{\mu}}{\ln\hat{\mu}^{-1}}+O([\tau\omega\hat\mu])
\quad \mbox{as} \quad \mu\to \mu(p),\quad ,$$ where
$\hat{\alpha}_1(p)>0$ and
$$\tau=\frac{1}{\ln\hat{\mu}^{-1}},\quad \quad
\omega=\frac{\ln \ln \hat{\mu}^{-1}}{\ln\hat{\mu}^{-1}},\quad
\hat{\mu}=\mu-\mu(p).$$
\item[{\rm (iii)}]
Let $\varphi(q_0(p))=\nabla\varphi(q_0(p))=0$. Then
$$m(p)-E(\mu,p)=a(p)\hat{\mu}+O([\tau\hat{\mu}])
\quad  as\quad{\mu}\to {\mu(p)}.$$
$$
\tau=\hat{\mu}\ln\hat{\mu},\quad\hat{\mu}={\mu}-{\mu}(p),$$ where
$a(p)>0.$
\end{enumerate}
\end{corollary}

\section{Proof of the results}

The parametrical Morse lemma and Hypothesis \ref{Hyp1} yield the
existence, for each $p\in U_{\delta}(0)$, a map $s=\psi(y,p)$ of the
sphere $W_{\gamma}(0)\subset \mathbb{R}^2$ to a neighborhood
$U(q_0(p))$ of the point $q_0(p)=(q_1^0(p), q_2^0(p))\in \T^2$ such
that the function $w_p(\psi(y,p))$ can be represented as
$$w_p(\psi(y,p))=m(p)+y^2.$$ Here the function $\psi(y,\cdot)$
resp. $\psi(\cdot,p)$ is holomorphic in $U_{\delta}(0)$ resp.
$W_{\gamma}(0)$ and $\psi(0,p)=q_0(p)$. Moreover, the Jacobian
$J(\psi(y,p))$ of the mapping  $s=\psi(y,p) $ is analytic in
$W_{\gamma}(0)$ and positive, i.e.
\begin{equation}\label{Jacobian}
J(\psi(y,p))= \left\|\begin{matrix}\dfrac{\partial\psi_1}{\partial
y_1}(y,p)&\dfrac{\partial\psi_1}{\partial y_2}(y,p)\\
\dfrac{\partial\psi_2}{\partial
y_1}(y,p)&\dfrac{\partial\psi_2}{\partial
y_2}(y,p)\end{matrix}\right\|>0
\end{equation}\label{delta}
for all $p\in U_{\delta}(0)$ and $y\in W_{\gamma}(0)$.

Now we establish an expansion for $\Delta(\mu,p;z)$ in the
half-neighborhood $(m(p)-\varepsilon,m(p))$ of the point $z=m(p)$,
which plays an important role in the proof of the main results.
\begin{lemma}\label{det_expansion}
Assume Hypothesis \ref{Hyp1}. Then for any sufficiently small
$m(p)-z>0$ the function $\Delta(\mu,p; \cdot),\,\mu>0,\,p\in
U_{\delta}(0)$ can be represented as the following convergent
series:
\begin{itemize}
\item[(i)] if $\varphi(q_0(p))\neq 0$, then
\begin{align*}\label{30}\Delta(\mu,p;z)=1-\mu \alpha_0(p)\ln(m(p)-z)-
\mu\ln(m(p)-z)\\ \nonumber
\times\sum\limits_{n=1}^{\infty}\alpha_n(p)\left({m(p)-z}\right)^n-
\mu F(p,z),\end{align*} where $\ln(\cdot)$ is the branch of function
$\mathrm{Ln}$ assuming the real values for $m(p)-z>0$,
$\alpha_0(p)=-\frac{1}{2}\varphi^2(q_{0}(p))J(q_{0}(p))$, the
coefficients $\alpha_1(p), \alpha_2(p),...$ are real numbers and

\begin{equation*}\label{31}
F(\mu,z)=\sum\limits_{n=0}^{\infty}c_n(p)\left({m(p)-z}\right)^{n},\end{equation*}
with real coefficients $c_0(p), c_1(p), c_2(p),...$.

\item[(ii)] if $\varphi(q_0(p))=0$ and
$\nabla\varphi(q_0(p))=\left(\dfrac{\partial\varphi}{\partial
q_1}(q_0(p), \dfrac{\partial\varphi}{\partial
q_2}(q_0(p)\right)\neq0$, then
\begin{equation}\label{delta}\Delta(\mu,p;z)=1-\frac{\mu}{\mu(p)}-\mu
\ln(m(p)-z)\sum\limits_{n=1}
^{\infty}\hat{\alpha}_n(p)\left({m(p)-z}\right)^n- \mu
\hat{F}(p,z),\end{equation} where
\begin{align}\label{qqq1}\hat{\alpha}_1(p)=\frac\pi
2J(q_0(p)) \left\{\left[\dfrac{\partial\varphi}{\partial
q_1}(q_0(p))\dfrac{\partial\psi_1}{\partial
y_1}(0,p)+\dfrac{\partial\varphi}{\partial
q_2}(q_0(p))\dfrac{\partial\psi_2}{\partial
y_1}(0,p)\right]^2\right.\\ \nonumber+
\left.\left[\dfrac{\partial\varphi}{\partial
q_1}(q_0(p))\dfrac{\partial\psi_1}{\partial
y_2}(0,p)+\dfrac{\partial\varphi}{\partial
q_2}(q_0(p))\dfrac{\partial\psi_2}{\partial
y_2}(0,p)\right]^2\right\}>0,\end{align} $\hat{\alpha}_n(p),
n=2,3,...$ are real numbers and
\begin{equation*}\label{ww}
\hat{F}(p,z)=\sum\limits_{n=1}^{\infty}\hat{c}_n(p)\left({m(p)-z}\right)^{n}
\end{equation*}
with real coefficients $\hat{c}_n(p), n=1,2,...$. \end{itemize}
\end{lemma}
\begin{remark}
We remark that, if in the part (ii) of Lemma \ref{det_expansion} we
assume that  $\varphi(q_0(p))=0$ and $\nabla\varphi(q_0(p))=0$, all
results of Lemma \ref{det_expansion} are hold, except the inequality
$\hat{\alpha}_1(p)>0$.
\end{remark}

\begin{proof}[Proof of Lemma \ref{det_expansion}]
The part (i) of Lemma \ref{det_expansion} is proven as Lemma 3.6 in
\cite{LAK2012}.  Therefore, we  prove here only the part (ii).

 We represent the function \eqref{omega} in the form
\begin{equation}\label{form}{\Omega}(p,z)={\Omega}_1(p,z)+{\Omega}_2(p,z),\end{equation}
where
\begin{gather} \label{first}{\Omega}_1(p,z)
=\int\limits_{U(q_0(p))}\dfrac{\varphi^2(s)ds}{w_p(s) -z},\quad
{\Omega}_2(p,z)=\int\limits_{\T^2\setminus U(q_0(p))}
\dfrac{\varphi^2(s)ds}{w_p(s) - z},
\end{gather}
and $U(q_0(p))\in \T^2$ is a sufficiently small neighborhood of the
minimum point $q_0(p)\in \T^2$.

Since $m(p)$ is a unique minimum of the function $w_p(s)$ one
concludes that for any $p\in U_{\delta}(0)$ the function
${\Omega}_2(p,z)$ is analytic in some neighborhood of the point
$z=m(p)$.

Taylor's series expansion of the function $\varphi(q)$ in the
neighborhood $U(q_0(p))$ of $q=q_0(p)=(q_1^{0}(p), q_2^{0}(p))$ can
be written as
\begin{align*}\nonumber\varphi^2(q)=a_1^2(p)(q_1-q_1^{0}(p))^2
+a_2^2(p)(q_2-q_2^{0}(p))^2+2a_1(p)
a_2(p)(q_1-q_1^{0}(p))(q_2-q_2^{0}(p))\\+\sum\limits_{n=3}^{\infty}\sum\limits_{\,i_1,i_2,..,i_n=1}^{2}
a_{i_1i_2..i_n}(p)\prod\limits_{k=1}^{n}(q_{i_k}-q^{0}_{i_k}(p)),\end{align*}
where $a_1(p)=\dfrac{\partial \varphi}{\partial q_1}(q_0(p)),\,
a_2(p)=\dfrac{\partial \varphi}{\partial q_2}(q_0(p))$\,and
$a_{i_1i_2..i_n}(p),\,i_1,i_2,...,i_n=1,2,...$ are real numbers.
Hence and by \eqref{first} we have

\be \label{m}\Omega_1(p,z)=\Omega_{11}(p,z)+
\Omega_{12}(p,z)+\Omega_{13}(p,z)+\Omega_{14}(p,z),\ee where
\begin{equation}\label{11}
\Omega_{11}(p,z)=a_1^2(p)\int\limits_{U(q_0(p))}\dfrac{(s_1-q_1^0(p))^2ds}{w_p(s)
-z},\,\,\,
\Omega_{12}(p,z)=a_2^2(p)\int\limits_{U(q_0(p))}\dfrac{(s_2-q_2^0(p))^2ds}{w_p(s)
-z},\end{equation} \begin{equation}\nonumber
\Omega_{13}(p,z)=2a_1(p)a_2(p)\int\limits_{U(q_0(p))}\dfrac{(s_1-q_1^0(p))(s_2-q_2^0(p))ds}{w_p(s)
-z},\end{equation} \begin{equation}\label{3.}
\Omega_{14}(p,z)=\sum\limits_{n=3}^{\infty}\sum\limits_{\,i_1,i_2,..,i_n=1}^{2}
a_{i_1i_2..i_n}(p)\int\limits_{U(q_0(p))}\prod\limits_{k=1}^{n}\frac{(s_{i_k}-q^{0}_{i_k}(p))}{w_p(s)
-z}ds.\end{equation}

Making a change of variables $s=\psi(y,p)$, in the first integral of
\eqref{11}, gives
\begin{equation}\label{omega11}\Omega_{11}(p,z)=a_1^2(p)\int\limits_{W_{\gamma}(0)}\dfrac{(\psi_1(y,p)-q_1^0(p))^2
J(\psi(y,p))dy}{y^2+m(p)-z}.\end{equation} The regularity of the
functions $\psi(y,p)$ and $J(\psi(y,p))$ in $W_{\gamma}(0)$ yields
the following expansions
$$\psi_1(y,p)=q_1^0(p)+b_1(p)y_1+b_2(p)y_2+\sum\limits_{k,l\in N, k+l=4}^{\infty}
b_{kl}(p)y_1^{k-1}y_2^{l-1},$$
$$J(\psi(y,p))=J(q_0(p))+\sum\limits_{k,l\in N,
k+l=3}^{\infty}d_{kl}(p)y_1^{k-1}y_2^{l-1},  $$ and consequently
\begin{align}\label{Omega1}
(\psi_1(y,p)-q_1^0(p))^2J(\psi(y,p))=b_1^2(p)J(q_0(p))y_1^2+b_2^2(p)J(q_0(p))y_2^2
\\ \nonumber+2b_1(p)b_2(p)J(q_0(p))y_1y_2+\sum\limits_{k,l\in N,
k+l=5}^{\infty}g_{kl}(p)y_1^{k-1}y_2^{l-1},\end{align} where
$$b_1(p)=\frac{\partial\psi_1}{\partial y_1}(0,p),\quad
b_2(p)=\frac{\partial\psi_1}{\partial y_2}(0,p),$$ and $g_{k
l}(p),\,\, k,l\in N$ are real numbers. By \eqref{omega11}, for
 $\Omega_{11}(p,z)$ we get the representation
\begin{equation}\label{Omega2}\Omega_{11}(p,z)=\Omega_{11}^{(1)}(p,z)+\Omega_{11}^{(2)}(p,z)+\Omega_{11}^{(3)}(p,z)
+\Omega_{11}^{(4)}(p,z),
\end{equation} where
\begin{align}\label{Omega3}&\Omega_{11}^{(1)}(p,z)=a_1^2(p)b_1^2(p)J(q_0(p))
\int\limits_{W_{\gamma}(0)}\dfrac{y_1^2dy}{y^2+m(p)-z},\\
&\Omega_{11}^{(2)}(p,z)=a_1^2(p)b_2^2(p)J(q_0(p))
\int\limits_{W_{\gamma}(0)}\dfrac{y_2^2dy}{y^2+m(p)-z},\nonumber\\
&\Omega_{11}^{(3)}(p,z)=2
a_1^2(p)b_1(p)b_2(p)J(q_0(p))\int\limits_{W_{\gamma}(0)}\dfrac{y_1y_2dy}{y^2+m(p)-z}\nonumber
\end{align}
and
\begin{equation}\label{Omega}
\Omega_{11}^{(4)}(p,z)=\sum\limits_{k,l\in N,
k+l=5}^{\infty}g_{kl}(p)
\int\limits_{W_{\gamma}(0)}\dfrac{y_1^{k-1}y_2^{l-1}dy}{y^2+m(p)-z}.
\end{equation}
Passing the polar coordinates as  $y_1=r\cos\alpha,$
$y_2=r\sin\alpha,$ $0\leq r\leq \gamma,$
 $0\leq\alpha\leq2\pi$ in \eqref{Omega3} yields
\begin{equation}\label{Omega4}\Omega_{11}^{(1)}(p,z)=\pi a_1^2(p)b_1^2(p)J(q_0(p))
\int\limits_0^\gamma\dfrac{r^3dr}{r^2+m(p)-z}.\end{equation} Recall
that for any $\zeta<0$ and $n\in N$ the following equality
\begin{equation}\label{propos} I_{n}(\zeta) =\int\limits_{0}^\delta \frac{r^{2n+1}dr}{r^2
-\zeta}= -\dfrac12\zeta^n\ln(-\zeta)+\hat{I}_n(\zeta)
\end{equation} holds, where $\hat{I}_n(\zeta)$ is a regular function in
some neighborhood of the origin \cite[Lemma 5]{Lakaev92}. Hence and
\eqref{propos} yields
\be\label{propos1}\Omega_{11}^{(1)}(p,z)=\frac{\pi}{2}
a_1^2(p)b_1^2(p)J(q_0(p))
(m(p)-z)\ln(m(p)-z)+\sum\limits_{n=0}^\infty \xi_n(p)(m(p)-z)^n,\ee
where  $\xi_n(p),n=1,2,...$ are real numbers. Analogously, for
$\Omega_{11}^{(2)}(p,z)$, we have the expansion
\begin{equation}\label{7}\Omega_{11}^{(2)}(p,z)=\frac{\pi}{2}
a_1^2(p)b_2^2(p)J(q_0(p))
(m(p)-z)\ln(m(p)-z)+\sum\limits_{n=0}^\infty
\eta_n(p)(m(p)-z)^n\end{equation} and the equality
$\Omega_{11}^{(3)}(p,z)=0.$ Passing to the polar coordinates as
$y_1=r\cos\alpha,$ $y_2=r\sin\alpha,$ $0\leq r\leq \gamma,$
 $0\leq\alpha\leq2\pi$  in \eqref{Omega}, we get
\begin{equation*}\label{q11}\Omega_{11}^{(4)}(p,z)=\sum\limits_{k,l\in N,
k+l=5}^{\infty}g_{kl}(p)
\int\limits_0^{\gamma}\dfrac{r^{k+l-1}dr}{r^2+m(p)-z}\int\limits_0^{2\pi}
\cos^{k-1}\alpha\sin^{l-1}\alpha d\alpha.\end{equation*} Note that,
the integral
\begin{equation*}\label{cos}\int\limits_0^{2\pi}
\cos^{k-1}\alpha\sin^{l-1}\alpha d\alpha\end{equation*} is not equal
to zero in case  $k=2n_1+1$, $n_1\in N$ and  $l=2n_2+1,$ $n_2\in N$
and is equal to zero in case $k=2n_1,\,n_1\in N$ or \,
$l=2n_2,\,n_2\in N.$ According to this the function
$\Omega_{11}^{(4)}(p,z)$  is represented as
$$\Omega_{11}^{(4)}(p,z)=\sum\limits_{n=2}^{\infty}\hat{g}_{n}(p)
\int\limits_0^{\gamma}\dfrac{r^{2n+1}dr}{r^2+m(p)-z}.$$ The equality
\eqref{propos} yields
\begin{equation}\label{9}\Omega_{11}^{(4)}(p,z)=\ln(m(p)-z)\sum\limits_{n=2}^{\infty}q_{n}(p)(m(p)-z)^n+
\sum\limits_{n=0}^\infty\theta_n(p)(m(p)-z)^n,\end{equation} where
$q_{n}(p)$ and $\theta_n(p)$ are real numbers. Taking into account
\eqref{propos1}, \eqref{7}, \eqref{9} and \eqref{Omega2} we have the
following expansion
$$\Omega_{11}(p,z)=\ln(m(p)-z)\sum\limits_{n=1}^\infty
d_n(p)(m(p)-z)^n+\sum\limits_{n=0}^\infty \hat{d}_n(p)(m(p)-z)^n,$$
where $$d_1(p)=\frac{\pi}{2}a_1^2(p)J(q_0(p))(b_1^2(p)+b_2^2(p)).$$
Analogously it is found the expansions for functions
$\Omega_{12}(p,z)$ and $\Omega_{13}(p,z),$ i.e.
$$ \Omega_{12}(p,z)=\ln(m(p)-z)\sum\limits_{n=1}^\infty
e_n(p)(m(p)-z)^n+\sum\limits_{n=0}^\infty \hat{e}_n(p)(m(p)-z)^n,$$
$$\Omega_{13}(p,z)=\ln(m(p)-z)\sum\limits_{n=1}^\infty
f_n(p)(m(p)-z)^n+\sum\limits_{n=0}^\infty \hat{f}_n(p)(m(p)-z)^n.$$
Here $e_n(p),$ $\hat{e}_n(p),$ $f_n(p)$ and $\hat{f}_n(p)$ are real
numbers with
$$ e_1(p)=\frac{\pi}{2}a_2^2(p)J(q_0(p))(l_1^2(p)+l_2^2(p)),$$ $$
f_1(p)=\pi a_1(p)a_2(p)J(q_0(p))(b_1(p)l_1(p)+b_2(p)l_2(p)),$$ where
$l_1(p)=\dfrac{\partial\psi_2}{\partial y_1}(0,p)$ and
$l_2(p)=\dfrac{\partial\psi_2}{\partial y_2}(0,p)$.

Making a change of variables $s=\psi(y,p)$ in \eqref{3.}  we find
$$\Omega_{14}(p,z)=\sum\limits_{n=3}^{\infty}\sum\limits_{\,i_1,i_2,..,i_n=1}^{2}
a_{i_1i_2..i_n}(p)\int\limits_{W_{\gamma}(0)}\prod\limits_{k=1}^{n}
\frac{(\psi_{i_k}(y,p)-q^{0}_{i_k}(p))J(\psi(y,p))}{y^2+m(p)
-z}dy.$$ Expanding the function
$(\psi_{i_k}(y,p)-q^{0}_{i_k}(p))J(\psi(y,p))$ at the point $y=0$,
as in  \eqref{Omega1}, from the last equation we obtain
$$\Omega_{14}(p,z)=\sum\limits_{k,l\in N,
k+l=5}^{\infty}{\hat{a}}_{kl}(p)\int\limits_{W_{\gamma}(0)}
\frac{y_1^{k-1}y_2^{l-1}}{y^2+m(p) -z}dy,$$
${\hat{a}}_{kl}(p),\,k,l\in N$ are real numbers. Hence, similarly as
above we get the following expansion

$$\Omega_{14}(p,z)=\ln(m(p)-z)\sum\limits_{n=2}^{\infty}g_{n}(p)(m(p)-z)^n+
\sum\limits_{n=0}^\infty \hat{g}_n(p)(m(p)-z)^n,$$  where ${g}_n(p)
$ and  $\hat{g}_n(p) $ are real numbers

The expansions of functions $\Omega_{11}(p,z),$ $\Omega_{12}(p,z),$
$\Omega_{13}(p,z),$ $\Omega_{14}(p,z)$ and the equalities
\eqref{form}, \eqref{m} give
\begin{equation}\label{w}\Omega(p,z)=\ln(m(p)-z)\sum\limits_{n=1}^{\infty}\hat{\alpha}_n(p)(m(p)-z)^n+
\sum\limits_{n=0}^\infty\hat{c}_n(p)(m(p)-z)^n,\end{equation}
$$\hat{\alpha}_1(p)=\frac\pi 2J(q_0(p))
\big[(a_1(p)b_1(p)+a_2(p)l_1(p))^2+(a_1(p)b_2(p)+a_2(p)l_2(p))^2\big].$$
Now we proof that $\hat{\alpha}_1(p)\neq 0.$ Assume the converse,
i.e.,
\begin{equation}\label{system}
\begin{cases}a_1(p)b_1(p)+a_2(p)l_1(p)=0\\
a_1(p)b_2(p)+a_2(p)l_2(p)=0.\end{cases}
\end{equation} Since,
$a_1(p)=\dfrac{\partial \varphi}{\partial q_1}(q_0(p))$ and
$a_2(p)=\dfrac{\partial \varphi}{\partial q_2}(q_0(p))$ by the
assumption (ii) of Lemma \ref{det_expansion}, at least one of these
two numbers $a_1(p)$ or $a_2(p)$ is not zero. Assume $a_1(p)\neq0$.
Then, multiplying the first equation of the system \eqref{system} by
$l_2(p)$, the second one by $l_1(p)$ and also subscribing them term
by term we obtain the equality $$b_1(p)l_2(p)=b_2(p)l_1(p),$$ which
contradicts to the inequality \eqref{Jacobian}. Thus
$\hat{\alpha}_1(p)>0$. Passing to the limits as  $z\to m(p)$ in
\eqref{w}, we get $\hat{c}_0(p)=1/\mu(p)$, which completes the prove
of Lemma \ref{det_expansion}.
\end{proof}
The part (ii) of Lemma \ref{det_expansion} yields
\begin{corollary}If $\varphi(q_0(p))=\nabla\varphi(q_0(p))=0$,
then $\hat{\alpha}_1(p)=0.$
\end{corollary}

\begin{lemma}\label{koef}Assume Hypothesis \ref{Hyp1}, $p\in U_{\delta}(0)$ and $\varphi(q_0(p))=0$.
Then for coefficients $\hat{\alpha}_1(p),$ $\hat{c}_1(p)$ of the
expansion \eqref{delta} the following relation holds:
$$|\hat{\alpha}_1(p)|+|\hat{c}_1(p)|\neq 0.$$
\end{lemma}
\begin{proof}[Proof] Assume the converse that is
$\hat{\alpha}_1(p)=\hat{c}_1(p)=0.$ Then Lemma \ref{egenvalue} and
the equality \eqref{delta} yields
\begin{align}\label{c}
-\frac{\mu-\mu(p)}{\mu(p)}-\mu \ln(m(p)-E(\mu,p))\sum\limits_{n=2}
^{\infty}\hat{\alpha}_n(p)\left({m(p)-E(\mu,p)}\right)^n\\
\nonumber-\mu
\sum\limits_{n=2}^{\infty}\hat{c}_n(p)\left({m(p)-E(\mu,p)}\right)^n=0.\end{align}
For each $p\in \T^2$ the eigenvalue $E(\mu,p)$ is concave function
of $\mu\ge0$ (see \cite[Theorem 1]{Bareket:1981}). Since every
concave function on $\mathbb{R}$ has a finite right derivatives the
finite limit
\begin{equation}\nonumber
\lim\limits_{\mu\to \mu(p)^+} \dfrac{E(\mu,p)-m(p)}{ \mu - \mu(p)}
\end{equation}
 exists. Hence we get
$$m(p)-E(\mu,p)=C(\mu-\mu(p))+o(\mu-\mu(p)),\quad \mu\to\mu(p),\quad 0\leq C<\infty. $$
Consequently, using \eqref{c} we have
\begin{align}\nonumber\label{c}-\mu
\ln[C(\mu-\mu(p))+o(\mu-\mu(p))]\sum\limits_{n=2}
^{\infty}\hat{\alpha}_n(p)\left[C(\mu-\mu(p))+o(\mu-\mu(p))\right]^{n-1}\\
\nonumber-\mu \sum\limits_{n=2}^{\infty}\hat{c}_n(p)
\left[C(\mu-\mu(p))+o(\mu-\mu(p))\right]^{n-1}=\frac{1}{\mu(p)},\,
as\,\, \mu\to\mu(p).\end{align} Passing to the limits as
$\mu\to\mu(p)$ in both sides of the last equation we have
$1/{\mu(p)}=0.$ This contradiction show that
 $|\hat{\alpha}_1(p)|+|\hat{c}_1(p)|\neq0.$
\end{proof}

Now we are able to prove the main results.

\begin{proof}[\bf Proof of Theorem \ref{main}]
Set $\mu(p,z)=(\Omega(p,z))^{-1}>0$, $p\in U_{\delta}(0),\,
z\in(-\infty,m(p))$. The function
 $\mu(p,\cdot):$ $z\in (-\infty,
m(p)) \mapsto \mu \in (\mu(p),+\infty)$ is continuous and monotone
decreasing in $z\in (-\infty, m(p))$. Then
$$\lim\limits_{z\to m(p)-0} \mu(p,z)= \mu(p)\ge0.$$
Therefore $\mu(p,z)$ has a continuous inverse $E(\cdot,p)$: $\mu \in
(\mu(p),+\infty)\mapsto z\in (-\infty, m(p))$. Clearly, $\Delta(\mu,
p\ ; E(\mu,p))\equiv 0.$ The above arguments will lead to a logical
conclusion  proving that
 $E(\cdot,p)\to m(p)-0,$ if and only if
$\mu\to \mu(p)+0.$

 Using appropriate changes of variables, one can reduce the proofs of parts (i) and
(ii) to the proof of the implicit function theorem for several
variables (see, e.g., \cite{{LKhLtmf12}},\cite{LDK2013}).

Therefore, we prove part (iii) of Theorem \ref{main}.

Denote by $\hat{\mu}=\mu-\mu(p)$ and $\alpha=m(p)-E(p,z)$, where
$\mu(p)>0$.

\item[(iii)] Let $\varphi(q_0(p))=0$ and $\nabla \varphi(q_0(p))=0.$
Then Lemma \ref{det_expansion} implies $\hat{\alpha}_1(p)=0$ and
Lemma \ref{koef} gives
 $\hat{c}_1(p)\neq0.$ Hence Lemmas \ref{det_expansion} and
 \ref{egenvalue}
yield the equation
\begin{equation}\label{ass}-\frac{\hat{\mu}}{\hat{\mu}\mu(p)+\mu^2(p)}=
\ln\alpha\sum\limits_{n\geq 2}^{\infty}\hat{\alpha}_n(p){\alpha}^n+
\sum\limits_{n\geq 1}^{\infty}\hat{c}_n(p){\alpha}^n,\end{equation}
from which one can see that $\hat{c}_1(p)<0.$

 Introducing now the variables
\begin{equation*}\label{gg}\alpha=\hat{\mu}\,(a(p)+u),\quad
a(p)=[-\hat{c}_1(p)\,\mu^2(p)]^{-1},\end{equation*} yields that in
the region, where $|\alpha|$ is small, the equation \eqref{ass} is
equivalent to
\begin{align}\label{ppp}\nonumber F(u,\tau, \hat{\mu})=
\left[ \tau+\hat{\mu}\ln(a(p)+u)\right] \times\sum\limits_{n\geq
2}\hat{\alpha}_n(p)\hat{\mu}^{n-2}(a(p)+u)^{n}\\+\sum\limits_{n\ge
1}\hat{c}_n(p)\hat{\mu}^{n-1}(a(p)+u)^{n}+\frac{1}{\hat{\mu}\mu(p)+\mu^2(p)}=0\end{align}
with $\tau=\hat{\mu}\ln\hat{\mu}.$ The function $F$ satisfies the \
following conditions:
\item[(i)] $u=0,$
  $\tau=0,$ $\hat{\mu}=0$ is a solution of
  $F(\cdot,\cdot,\cdot)=0$
\item[(ii)] $F$ is
analytic function for small $|u|$,  $|\tau|,$  $|\hat{\mu}|;$
\item[(iii)]
$\partial F/\partial u (0,0,0)=\hat{c}_1(p)\ne 0.$

The implicit function theorem yields that for all sufficiently small
$\tau$ and $\hat{\mu}$, equation \eqref{ppp} has a unique solution
$u(\hat{\mu},\tau)$, which is given by the absolutely convergent
series
$$u=\sum\limits_{l, s\ge0} c(l, s)\, \tau^{l}\,\hat{\mu}^{s}.$$

The condition $\tau=\hat{\mu}=0$ gives the equality ${c}(0,0)=0.$
Given \eqref{gg}, we have
\begin{align}\nonumber \alpha = a(p)\hat{\mu} + \sum\limits_{l,s\ge
0, l+s\geq1} c(l,s)\,\tau^{l}\,\hat{\mu}^{s+1}=a(p)\hat{\mu} +
\sum\limits_{l\geq0,s\ge1, l+s\geq2}
\hat{c}(l,s)\,\tau^{l}\,\hat{\mu}^{s},\end{align} which yields
\eqref{ox}.
\end{proof}
{\bf Acknowledgments} This research was supported by the Foundation
for Basic Research of the Republic of Uzbe\-kistan (Grant
No.OT-F4-66).

\end{document}